\documentclass[12pt,a4paper]{amsart}
\usepackage{amssymb,amsmath,amsthm}
\usepackage{epsfig}

\newtheorem{theorem}{Theorem}[section]
\newtheorem{lemma}{Lemma}[section]
\newtheorem{corollary}{Corollary}[section]

\newcommand{\om}{\omega}

\newcommand{\zz}{{\mathbb{Z}^d}}
\newcommand{\rr}{\mathbb{R}^d}
\newcommand{\R}{\mathbb{R}}
\newcommand{\N}{\mathbb{N}}

\newcommand{\conv}{\mathrm{conv}}
\newcommand{\vv}{\mathrm{vert\;}}
\newcommand{\vol}{\mathrm{vol\;}}
\newcommand{\p}{\mathcal{P}}
\newcommand{\aff}{\textrm{aff\;}}

\newcommand{\tri}{\Delta}

\numberwithin{equation}{section}

\begin{document}

\title{Volumes of convex lattice polytopes and a question of V. I. Arnold}
\author{Imre B\'ar\'any and Liping Yuan}

\begin{abstract}
We show by a direct construction that there are at least $\exp\{cV^{(d-1)/(d+1)}\}$ convex lattice
polytopes in $\rr$ of volume $V$ that are different in the sense that none of them can be
carried to an other one by a lattice preserving affine transformation. This is achieved by
considering the family $\p^d(r)$  (to be defined in the text) of convex lattice polytopes whose volumes are
between $0$ and $r^d/d!$. Namely we prove that for $P \in \p^d(r)$, $d!\vol P$ takes all possible integer values
between $cr^{d-1}$ and $r^d$ where $c>0$ is a constant depending only on $d$.
\end{abstract}

\maketitle

\section{Introduction and main result}\label{sec:introd}

Let $e_1,\ldots,e_d$ be the standard basis of $\R^d$, $d \ge 2$ and for $r \in \N$ define
\[
A(r)=\conv\{re_1,\ldots,re_d\} \mbox{ and } S(r)=\conv\{A(r)\cup \{0\}\}.
\]
We will write $A^d(r)$ and $S^d(r)$ in case of need. We further
define a family $\p(r)=\p^d(r)$ as the collection of all convex
lattice polytopes $P$ with $A(r)\subset P\subset S(r)$. As is well
known, $v(P)=d! \vol P$ is an integer for $P \in \p(r)$, and $0\le
v(P) \le r^d$ since $A(r)$ and $S(r)$ belong to the family $\p(r)$.
It is also clear that if $v(P)\ne 0$, then it is at least $r^{d-1}$
since $P\in \p(r)$ with $v(P)>0$ contains $d$-dimensional simplex
whose base is $A(r)$.

The question addressed in this paper concerns the set of values of
$\{v(P):P \in \p^d(r)\}$. It has emerged in connection with a
problem of V. I. Arnold~\cite{ar} (see also~\cite{bar}) as we will
explain soon. Our main result is

\begin{theorem}\label{th:main} Given $d\ge 2$ there is a number $c=c_d>0$ such that for all
integers $v \in [cr^{d-1},r^d]$ there is a polytope $P \in \p^d(r)$ with $v(P)=v$.
\end{theorem}

We will prove this result in a more precise form in Section~\ref{sec:proof}. The case $d=2$ is
simple: it is easy to see that $\{v(P):P \in \p^2(r)\}=\{0,r,r+1,\ldots,r^2\}$.
We will show in Section~\ref{sec:gaps} that there are other gaps besides
$(0,r^{d-1})$ in the set $\{v(P): P\in \p^d(r)\}$ when $d\ge 3$.

\section{Arnold's question}

Two convex lattice polytopes are {\sl equivalent} if one can be carried to the other by a lattice
preserving affine transformation. This is an equivalence relation and equivalent polytopes
have the same volume. Let $N_d(V)$ denote the number of equivalence classes of convex lattice
polytopes in $\rr$ of volume $V$. (Of course, $d!V$ is a positive integer.) Arnold~\cite{ar} showed that
\[
V^{1/3} \ll \log N_2(V) \ll V^{1/3} \log V.
\]
After earlier results by Konyagin and Sevastyanov~\cite{ks}, the upper bound was improved
and extended to higher dimensions to
\[
\log N_d(V) \ll V^{(d-1)/(d+1)}
\]
by B\'ar\'any and Pach ~\cite{bp} (for $d=2$) and by B\'ar\'any and
Vershik ~\cite{bv} (for $d\ge 2$). The lower bound $\log N_d(V) \gg
V^{(d-1)/(d+1)}$ for all $d\ge 2$ has recently been proved
in~\cite{bar}. More information about Arnold's question can be found
in Arnold~\cite{ar}, B\'ar\'any~\cite{bar}, and Zong~\cite{zh} and
Liu, Zong~\cite{lz}.

We obtain the same lower bound as a direct and fairly simple application of Theorem~\ref{th:main}:

\begin{corollary}\label{cor:main} $V^{(d-1)/(d+1)} \ll \log N_d(V)$.
\end{corollary}

Some remarks are in place here about notation and terminology. A convex polytope $P\subset
\rr$ is a lattice polytope if its vertex set, $\vv P$ is a subset of $\zz$, the integer
lattice. The number of vertices of a polytope $P$ is denoted by $f_0(P)$. Throughout the
paper we use, together with the usual ``little oh'' and ``big Oh'' notation, the convenient
$\ll$ symbol, which means, for functions $f,g :\R_+ \to \R$, that $f(V) \ll g(V)$ if there
are constants $V_0>0$ and $c>0$ such that $f(V) \le cg(V)$ for all $V>V_0$. These constants,
to be denoted by $c,c_1,\dots,b,b_1\dots$ may only depend on dimension. The Euclidean norm
of the vector $x=(x_1,\dots,x_d)\in \rr$ is $\|x\|=\sqrt{x_1^2+\dots +x_d^2}$. $B^d$ denotes the
Euclidean unit ball of $\rr$, and $\vol B_d=\omega_d$. We write $[n]=\{1,2,\dots,n\}$.

The paper is organized as follows. The main result is proved in the next section.
The integer convex hull and some of its properties are
given in Section~\ref{sec:intconv}. The proof of Corollary~\ref{cor:main} is the content of
Section~\ref{sec:cor}. Then we describe further gaps in $\{v(P):P\in \p^r(r)\}$ when $d\ge 3$.
We finish with concluding remarks.

\section{Proof of Theorem~\ref{th:main}}\label{sec:proof}

As we have mentioned, $\{v(P):P \in \p^2(r)\}=\{0,r,r+1,\ldots,r^2\}$. We assume from now on that $d\ge 3$.
Here comes the more precise version of Theorem~\ref{th:main}.

\begin{theorem}\label{th:missed} Assume $d\ge 3$ and $r> r_0=2^dd!$. Then for every non-negative integer
$m\le (r-2^dd!)^d$ there is a $P \in \p^d(r)$ with $v(P)=r^d-m$.
\end{theorem}

This result implies Theorem~\ref{th:main} for $r> r_0$ with $c=2^dd!$ for instance. For $r\le r_0$ the theorem holds by
choosing the constant $c$ large enough.

It will be more convenient to work with
\[
m(P)=d!\vol(S^d(r)\setminus P)
\]
which we call the {\sl missed volume} of $P \in \p^d(r)$. (It would be more appropriate to call
it missed volume times $d!$ though.) With this notation Theorem~\ref{th:missed} says
that the set $\{m(P):P \in \p^d(r)\}$, that is, the set of missed volumes, contains all integers between
$0$ and $(r-2^dd!)^d$ provided $r>r_0$.

The proof is based on the following

\begin{lemma}\label{l:g-func} Assume $d\ge 2$ and let $g(x)=(2x)^d$ and $m \in \N$. Then there are integers $x_0\ge x_1 \ge \ldots \ge x_{d-1}\ge 0$
and an integer $m_d \in \{0,1,\ldots,2^dd!\}$ such that
\[
m=g(x_0)+g'(x_1)+g''(x_2)+\ldots +g^{(d-1)}(x_{d-1})+m_d.
\]
\end{lemma}

{\bf Proof.} We give an algorithm that outputs the numbers $x_0,\ldots,x_{d-1}$ and $m_0=m,m_1,\ldots,m_d$.

Start with $m_0=m$ and let $x_0$ be the unique non-negative integer with $g(x_0)\le m_0<g(x_0+1)$. If $m_{i-1}$ and $x_{i-1}$ have been defined then set $m_i=m_{i-1}-g^{(i-1)}(x_{i-1})$ and let $x_i$ be the unique non-negative integer with $g(x_i)\le m_i<g(x_i+1)$. We stop with $m_{d-1}$ and $x_{d-1}$ and define $m_d=m_{d-1}-g^{(d-1)}(x_{d-1})$.

We claim that $x_i\le x_{i-1}$. Note first that by construction and by the intermediate value theorem
\begin{eqnarray*}
m_i&=&m_{i-1}-g^{(i-1)}(x_{i-1})<g^{(i-1)}(x_{i-1}+1)-g^{(i-1)}(x_{i-1})\\
   &=&g^{(i)}(\xi)\le g^{(i)}(x_{i-1}+1),
\end{eqnarray*}
where $\xi \in [x_{i-1},x_{i-1}+1]$, and we also used that $g^{(i)}(x)$ is increasing for $x\ge 0$. So if, contrary to the claim,
we had $x_i > x_{i-1}$, then $x_{i-1}+1\le x_i$. As $g^{(i)}(x)$ is increasing we have
\[
m_i < g^{(i)}(x_{i-1}+1) \le g^{(i)}(x_i)\le m_i,
\]
a contradiction.

The same method gives that $m_d\le 2^dd!$:
\begin{eqnarray*}
m_d&=&m_{d-1}-g^{(d-1)}(x_{d-1})<g^{(d-1)}(x_{d-1}+1)-g^{(d-1)}(x_{d-1})\\
    &=&g^{(d)}(\xi)=2^dd!,
\end{eqnarray*}
for all $x$. The proof is finished by adding the defining equalities $m_i=m_{i-1}-g^{(i-1)}(x_{i-1})$ for $i=1,2,\ldots,d$.\hfill$\Box$

\medskip
{\bf Remark.} The same method works for every polynomial $g$ of degree $d$ such that $g^{(i)}(x)>0$ for all $i=0,1,\ldots,d$ and $x>0$.

\medskip
We return now to the {\bf proof} of Theorem~\ref{th:missed}. So given $r >r_0$ and $m \in \{0,1,\ldots,(r-2^dd!)^d\}$ we are going to construct $P \in \p(r)$ with $m(P)=m$. This is easy if $m \le r$: the simplex $\tri$ with vertices $me_1,e_2,\ldots,e_d$ has $v(\tri)=m$ so the missed volume of the closure of $S(r)\setminus \tri$, which is a lattice polytope, is $m$. So assume $m > r$.

Apply Lemma~\ref{l:g-func} with $m$ which is at most $(r-2^dd!)^d$, to get numbers $x_0,\ldots,x_{d-1}$ and $m_d$. Note that $2x_0 \le r-2^dd!$ and also $x_0\ge 2$ as $m>r\ge r_0=2^dd!$. Next we are going to define simplices $\tri_0,\ldots,\tri_d$ that are lattice polytopes contained in $S^d(r)$,
are pairwise internally disjoint, and $v(\tri_i)=g^{(i)}(x_i)$ for $i=0,1,\ldots,d-1$ and $v(\tri_d)=m_d$.

We set $e_i^*=2x_0e_i$ and define $\tri_0=\conv \{e_1^*,\ldots,e_d^*\}$, a non-degenerate simplex because $x_0\ge 2$. Clearly,
$v(\tri_0)=2^dx_0^d=g(x_0)$ and $\tri_0 \subset S(r)$.

Next, for $i \in [d]$, we let $\tri_i$ be the convex hull of vectors
\begin{eqnarray*}
&\;&e_i^*,e_i^*+2^dd(d-1)\ldots(d-i+1)e_i \mbox{ and }\\
&\;&e_i^*+(e_j-e_i)\; (j<i) \mbox{ and } e_i^*+x_i(e_j-e_i)\; (j>i),
\end{eqnarray*}
this is a simplex whose vertices are lattice points in $S(r)$, as one can check easily. Its edges starting from vertex $e_i^*$
are the vectors $e_j-e_i$ for $j<i$, $2^dd(d-1)\ldots(d-i+1)e_i$, and $x_i(e_j-e_i)$ for $j>i$. As the vectors
$e_1-e_i,\ldots,e_{i-1}-e_i,e_i,e_{i+1}-e_i,\ldots,e_d-e_i$ form a basis of the lattice $\zz$,
\[
v(\tri_i)=2^dd(d-1)\ldots(d-i+1)x_i^{(d-i)}=g^{(i)}(x_i).
\]
Note that $\tri_i$ may be degenerate (exactly when $x_i=0$) but only for $i>0$.

These simplices are internally disjoint. Indeed, $\tri_0$ lies on one side of the hyperplane $\aff\{e_1^*,\ldots,e_d^*\}$
and all other simplices are on the other side. Further, as $2x_0 \le r-2^dd!$ and $x_i \le x_0$,
every $\tri_i$ with $i>0$ is contained in the simplex whose vertices are $e_i^*,e_i^*+2^dd!e_i$ and
$\frac 12 (e_i^*+e_j^*)$ for $j\ne i$, and these larger simplices are internally disjoint.

We check next that the closure of $S(r) \setminus \bigcup_0^d \tri_i$ is a lattice polytope $P$ in $\p(r)$.
Write $X_0$ for the set of vertices of $\tri_0$ except $0$, and $X_i$ for the set of vertices of $\tri_i$ except $e_i^*$.
Let $Y$ be the set of vertices of $A(r)$.
$P$ is obtained from $S(r)$ by deleting $\tri_0,\tri_1,\ldots,\tri_d$ in this order. Deleting $\tri_0$ results
in $P_0=\conv (Y \cup X_0)\in \p(r)$. Deleting $\tri_1$ from $P_0$ gives $P_1=\conv (Y \cup X_0\cup X_1) \in \p(r)$. Similarly, deleting $\tri_i$ from $P_{i-1}$ results in a convex lattice polytope $P_i \in \p(r)$ whose vertices are $(Y \cup \bigcup_0^iX_i)$. This works even if
$\tri_i$ is degenerate; then $P_i=P_{i-1}$.

Finally we have $P=P_d \in \p(r)$ and the missed volume of $P$ is
\[
m(P)=\sum_0^d v(\tri_i)=\sum_0^{d-1}g^{(i)}(x_i-1)+m_d=m.
\]

This finishes the construction and gives a polytope $P \in \p(r)$ with $m(P)=m$ if $r\ge r_0$.
\hfill$\Box$

\medskip
{\bf Remark.} The same construction with minor modification works for all $m\le (r-2^dd)^d$ and $r\ge r_0=2^dd$.

\section{The integer convex hull}\label{sec:intconv}

Suppose $K \subset \rr$ is a bounded convex set. Its {\sl integer convex hull}, $I(K)$, is
defined as
\[
I(K)=\conv (K \cap \zz),
\]
which is a convex lattice polytope if nonempty.

To avoid some trivial complications we assume that $r$ is large enough, $r\ge r_d$ say.
One important ingredient of our construction is
\[
Q_r=I(rB^d)=\conv (\zz \cap rB^d).
\]
Trivially $\vol Q_r \le \om_dr^d$. It is proved in B\'ar\'any and Larman in \cite{bl} that
$\vol (rB_d \setminus Q_r) \ll r^{d\frac {d-1}{d+1}}$. The last exponent will appear so often
that we write $D=d\frac {d-1}{d+1}$. The number of vertices of $Q_r$ is estimated in
\cite{bl} as
\[
r^D \ll f_0(Q_r) \ll r^D.
\]

The vertices of $Q_r$ are very close to the boundary of $rB^d$. More precisely, we have the following estimate which is also used in~\cite{bl}.
\begin{lemma}\label{l:close} If $x$ is a vertex of $Q_r$, then $r-\|x\| \ll r^{-(d-1)/(d+1)}$.
\end{lemma}

{\bf Proof.} The set $rB^d \cap (2x-rB^d)$ is convex and centrally symmetric with center $x \in \zz$.
It does not contain any lattice point $z\ne x$: if it does, then both $z$ and $2x-z$ are lattice points in $rB^d$
and $x=\frac 12 \left(z+(2x-z)\right)$ is not a vertex of $Q_r$. By Minkowksi's classical theorem,
$\vol (rB^d \cap (2x-rB^d))< 2^d$. The estimate in the lemma follows from this by a simple computation.\hfill$\Box$

\medskip
We let $\rr_+$ denote the set of $x=(x_1,\ldots,x_d)\in \rr$ with $x_i\ge 0$ for every
$i\in[d]$. In the proof of Corollary~\ref{cor:main} we will consider $Q^r=Q_r \cap \rr_+$.
It is clear that $Q^r=I(rB^d \cap \rr_+)$ and further, that
\[
\vol \left((rB^d \cap \rr_+)\setminus Q^r\right) \ll r^D.
\]

Let $X$ be the set of those vertices $x=(x_1,\ldots,x_d)$ of $Q^r$ for which $x_i>0$ for all
$i \in [d]$. We claim that
\[
r^D \ll |X| \ll r^D
\]
Only the lower bound needs some explanation. The number of vertices of $Q_r$ with $x_i=0$ for some
$i \in [d]$ is less than $d$ times the number of vertices of $Q_r^{d-1}$ which is of order $r^{(d-1)(d-2)/d}=o(r^D)$.
Then $|X|\ge 2^{-d}\left((f_0(Q_r)-o(r^D)\right)$, so indeed $r^D \ll |X|$.

Lemma~\ref{l:close} says that $\|x\| \ge r-b_1r^{-(d-1)/(d+1)}$ where $b_1>0$ depends only on $d$.
Define $r_0=r-b_1r^{-(d-1)/(d+1)}$, so $X$ lies in the annulus $rB^d \setminus r_0B^d$.
Consequently all lattice points in $r_0B^d\cap \rr_+$ are contained in $Q^r\setminus X$
and so $I(r_0B^d\cap \rr_+)\subset I(Q^r\setminus X)$. The result from \cite{bl} cited above applies
to $I(r_0B^d \cap \rr_+)$ and gives that
\[
\vol \left((r_0B^d \cap \rr_+) \setminus I(r_0B^d \cap \rr_+)\right) \ll r_0^D\ll r^D
\]
as $r_0<r$. This implies that with a suitable constant $b>0$
\[
2^{-d}\omega_d r^d -br^D \le \vol I(Q^r\setminus X) \le 2^{-d}\omega_d r^d.
\]

After these preparation we are ready for the proof of Corollary~\ref{cor:main}.

\section{Proof of Corollary~\ref{cor:main}}\label{sec:cor}

Given a (large enough) number $V$ with $d!V \in \N$ we are going to construct many non-equivalent convex lattice polytopes whose
volume equals $V$. As a first step, we define $r$ via the equation
\[
V=2^{-d}\omega_d r-br^D.
\]

Consider $Q^r$ from the previous section. For $Z \subset X$ we define $Q(Z)=I(Q^r \setminus Z)$, $Q(Z)$ is a convex lattice polytope, it contains $Q(X)$.
So $V \le \vol Q(X) \le \vol Q(Z) \le 2^{-d}\omega_d r^d$. Set $m(Z)=\vol Q(Z) -V$. Then
\[
m(Z)=\vol Q(Z)-V \le 2^{-d}\omega_dr^d -V=br^D.
\]

Since $r^D\ll |X|$, the number of polytopes $Q(Z)$ is $2^{|X|} \ge \exp\{b_2r^D\}=\exp\{b_3 V^{(d-1)/(d+1)}\}$ with suitable positive constants $b_2,b_3$. This is what we need in Corollary~\ref{cor:main}. But the volumes of the $Q(Z)$ are larger than $V$. So we are going to cut off volume $\vol Q(Z)-V=m(Z)/d!$ from $Q(Z)$ so that what is left is a lattice polytope of volume exactly $V$. We will do so using Theorem~\ref{th:main} or rather Theorem~\ref{th:missed} the following way.

\smallskip
Set $\rho= \lfloor r/10 \rfloor \in \N$ and assume $r_d$ is so large that $\rho >r_0=2^dd!$.
Consider $A(\rho)$, $S(\rho)$ and $\p(\rho)$ from Section~\ref{sec:introd}. Theorem~\ref{th:missed} says that
given if $m=d!m(Z)\le (\rho-2^dd!)^d$, there is a polytope $P=P(Z) \in \p$ with $m(P)=m$. Then
\[
P^*(Z)=\left[Q(Z)\setminus S(r)\right] \cup P(Z)=Q(Z)\setminus \left[S(r) \setminus P(Z)\right]
\]
is a convex lattice polytope, and its volume is exactly $V$.

\smallskip
Now we have $\exp\{b V^{(d-1)/(d+1)}\}$ convex lattice polytopes $P^*(Z)$,
each of volume $V$. We claim that any one of them is equivalent to at most $d!$ other $P^*(W)$. This will
clearly finish that proof of $\log N_d(V) \gg V^{(d-1)/(d+1)}$.

To prove this note first that $P^*(Z)$  has a long edge on the segment $[0,re_i]$ for all $i\in [d]$.
This edge is of the form $E_i=E_i(Z)=[\alpha_i e_i,\lfloor r \rfloor e_i]$ and $\alpha_i\le \rho \le r/10$.
It is quite easy to check (we omit the details) that $P^*(Z)$ has exactly
these $d$ edges whose length is larger than $0.9r$. So if $P^*(Z)$ and $P^*(W)$ are equivalent, then the lattice preserving affine
transformation $T$ that carries $P^*(Z)$ to $P^*(W)$, has to map each $E_i(Z)$ to a uniquely determined $E_j(W)$.
As the lines containing $E_i(Z)$ all pass through the origin, $T(0)=0$ and so $T$ is a linear map.
Thus $T$ permutes the axes and is lattice preserving. Then it has to permute the vectors $e_1,\ldots,e_d$.
There are $d!$ such transformations. Consequently, $P^*(Z)$ is equivalent with at most $d!$ other
polytopes of the form $P^*(W)$.\hfill$\Box$

\medskip
{\bf Remark.} The above construction can be modified so that all $P^*(Z)$ are non-equivalent. Namely, set $t=\lfloor r \rfloor$ and
replace $Z\subset X$ by
\[
Z^0=Z\cup \bigcup_1^d \{te_i,(t-1)e_i,\ldots, (t-i+1)e_i\}.
\]
Set $Q(Z^0)=I(Q^r \setminus Z^0)$ and $P^*(Z^0)=\left[Q(Z^0)\setminus S(r)\right]\cup P(Z^0)$ where $P(Z^0) \in \p(\rho)$
is again chosen so that $\vol P^*(Z^0)=V$. The long edges of $P^*(Z^0)$ are almost the same $E_i(Z)$ except that this time each carries
a `marker', namely the last $i$ lattice points are missing from $E_i(Z)$. So if $P^*(Z^0)$ and $P^*(W^0)$ are equivalent, then the corresponding lattice preserving affine transformation has to be the identity.

\section{Gaps in $v(\p)$}\label{sec:gaps}

 Let $A_k=\{(x_1, x_2, \cdots, x_d): x_1+x_2+\cdots x_d=k,
x_i\in \{0, 1, 2, \cdots, \\ k \}, i\in [d] \}$, where $k=0, 1,
\cdots, r$. For any $P \in \p(r)$, if $P \cap A_{r-2} \neq
\emptyset$, then clearly $v(P)\geq 2 r^{d-1}$. Now suppose that $P
\cap A_{r-2} = \emptyset$ but $P\cap  A_{r-1} \neq \emptyset$.
Assume that $P=\conv \{ A(r)\cup B\}$, where $B \subset A_{r-1}$,
$B=\{b_1, b_2, \cdots, b_t\}$,
 $b_i=(b^i_1, b^i_2, \cdots, b^i_d)$,  $b^i_j
 \in \{0, 1, 2, \cdots, r-1 \}$ for all $i\in [t]$ and $j\in [d]$.
 For any $b_i, b_k \in B$,  let $q_{ik}=\max\{|b^i_j - b^k_j| :  j\in [d]\}$
  and  $q= \max\{q_{ik} : b_i, b_k\in B\}$.

\begin{theorem}\label{them-gapd}
For any $d\geq 3$ and $P \in \p(r)$, $v(P)\notin [r^{d-1}+1,
r^{d-1}+r^{d-2}-1]$.
\end{theorem}

\begin{proof}
 If $|B|=1$, then $P$ is actually a $d$-dimensional simplex
with base $A(r)$, which means that $v(P)=r^{d-1}$.

 If $|B|=2$, then suppose that $B=\{b_1, b_2\}$. In this case $q=q_{12}=\max \{|b^1_j-b^2_j|:  j\in [d] \}$ and $q \ge 1$ as $b_1\ne b_2$. Thus  we may assume
 without loss of generality, that $b^1_d - b^2_d= q\ge 1$. Let
 $P_1=\conv\{A(r)\cup \{b_1\} \}$, $P_2=\conv $  $\{re_1, re_2,
 \cdots,$ $re_{d-1}, b_1, b_2\}$. Since $b^1_d \neq b^2_d$,  $P_2$ is a $d$-dimensional
 simplex. It is easy to check that $P_1$ and $P_2$ are internally disjoint which implies  $v(P)=v(P_1)+v(P_2)$.
 Clearly, $v(P_1)=r^{d-1}$. Furthermore, we have $$v(P_2)= \left|\det\left(\begin{array}{ccccc}
    1 &  \ldots & 1 & 1 & 1\\
    re_1 & \ldots & re_{d-1} & b_1 & b_2
  \end{array}\right)\right|=|b^1_d-b^2_d|r^{d-2},$$
  which means that $v(P)=r^{d-1}+q r^{d-2}\geq r^{d-1}+r^{d-2}$. Clearly  $v(P)\geq
  r^{d-1}+r^{d-2}$ still holds for any $P \in \p(r)$ satisfying $|B|\geq
  3$. The proof is complete.
\end{proof}

 Now we consider the special case when $d=3$.
  Define a graph $G=(V(G), E(G))$ such that $V(G)=A_{r-1}$, $E(G)=\{ b_ib_k:
  |b^i_1 - b^k_1|+  |b^i_2 - b^k_2|+ |b^i_3 - b^k_3|=2 \}$. Clearly
$G$ is a triangular grid graph with boundary $\conv\{(r-1)e_1,
(r-1)e_2\}\cup \conv\{(r-1)e_1, (r-1)e_3\} \cup \conv\{(r-1)e_2,
(r-1)e_3\}$. Furthermore, it is not difficult to see that for $b_i,
b_k
  \in A_{r-1}$, $b_ib_k\in E(G)$ if and only if $q_{ik}=1$.

For any $b_i, b_k \in V(G)$, let  $l_{b_ib_k}$ denote the line
 determined by $b_i, b_k$. In the
hyperplane determined by $A_{r-1}$, let $H^+_{ik}$ denote the open
halfplane bounded by $l_{b_ib_k}$ such that $|H^+_{ik}\cap
\{(r-1)e_1, (r-1)e_2, (r-1)e_3\}|=2$ and $H^-_{ik}$ the open
halfplane bounded by $l_{b_ib_k}$ satisfying $|H^-_{ik}\cap
\{(r-1)e_1, (r-1)e_2, (r-1)e_3\}|=1$. Furthermore, let $d_G(b_i,
b_k)$ denote the graphic distance between $b_i$ and $b_k$, i.e., the
length of the shortest paths between $b_i$ and $b_k$. Clearly,
$d_G(b_i, b_k)=q_{ik}$.  For any $b_i\in A_{r-1}$, let
$D_{b_i}(s)=\{b_k: b_k\in A_{r-1}, d_G(b_i, b_k) \leq s\}$.
\begin{theorem} \label{theo-gap3-1}
 If  $d=3$, then for any $P \in \p(r)$,
 $v(P)\notin [r^{2}+r+2,
r^{2}+2r-1]$.
\end{theorem}

\begin{proof}
If $|B|=1$, clearly $v(P)=r^2$. If $|B|=2$, then by the proof of
Theorem \ref{them-gapd}
  we know that $v(P)\in \{r^2+kr: k=1, 2, \cdots, r-1\}$.  Now suppose that $|B|=
  t$ $(t\geq 3)$ and  $B=\{b_1, b_2, \cdots, b_t\}$, where
  $b_i=(b^i_1, b^i_2,  b^i_3)$, $b^i_j
 \in \{0, 1, 2, \cdots, r-1 \}$ for all $i\in [t]$ and
 $j\in [3]$.

 If $q\geq 2$, assume without loss of generality that $q_{12}=q
 \geq 2$, then we have
  $v(P)\geq v(\conv\{A(r)\cup
  \{ b_1, b_2\} \})$ $= r^{2}+qr$ $\geq r^{2}+2r$.

  Now suppose that $q=1$. Then any two points in $B$ are adjacent in $G$, which
  forces $t=3$ and $\conv B$ is a $2$-dimensional
  simplex homothetic to   $A(r)$. Furthermore,
  $\conv B=\theta_B+ \frac{\lambda_B}{r}A(r)$, where $\lambda_B=\pm 1$.
      Suppose without loss of generality that $l_{b_1b_2}\parallel l_{re_1,
   re_2}$ and the vectors $\overrightarrow{b_1b_2}$ and $\overrightarrow{re_1,re_2}$ have the same direction.
    If $b_3 \in H^-_{12}$, then $\lambda_B=1$ and
  $v(P)=v(\conv \{A(r)\cup \{ b_1, b_2\} \})$ $+v(\conv\{re_3, b_1, b_2,
  b_3\})$ $=r^2+r+1$. If $b_3 \in H^+_{12}$, then $\lambda_B=-1$ and
  $v(P)=v(\conv \{A(r)\cup \{ b_1, b_2\} \})$ $+v(\conv\{re_1,re_2, b_1,
  b_3\})$ $+v(\conv\{re_2, b_1, b_2,
  b_3\})$  $=(r^2+r)+r+1=r^2+2r+1$.

Combining the above discussions we see that there is no $v(P)$ of
$P\in \p(r)$ lying in the interval $[r^{2}+r+2, r^{2}+2r-1]$, and
the proof is complete.
 \end{proof}

\begin{theorem}\label{theo-gap3-2}
If $d=3$ and  $r\geq 6$, then  for any $P \in \p(r)$,  $v(P)\notin
[r^{2}+2r+5, r^{2}+3r-1]$.
\end{theorem}

\begin{proof} Now still suppose that $|B|=
  t$  and  $B=\{b_1, b_2, \cdots, b_t\}$, where
  $b_i=(b^i_1, b^i_2,  b^i_3)$, $b^i_j
 \in \{0, 1, 2, \cdots, r-1 \}$ for all $i\in [t]$ and
 $j\in [3]$.

 If $q\geq 3$, say, $q_{12}=q\geq 3$, then
  $v(P)\geq v(\conv\{A(r)\cup \{b_1, b_2\})$ $=r^2+qr \geq r^2+3r$.

If $q\leq 2$, then combining the proof of Theorem \ref{theo-gap3-1},
we only need to consider the cases when  $B$ satisfies $t\geq 3$ and
$q=2$.

\begin{figure}[htbp]
\centering
\includegraphics[height=4.5cm]{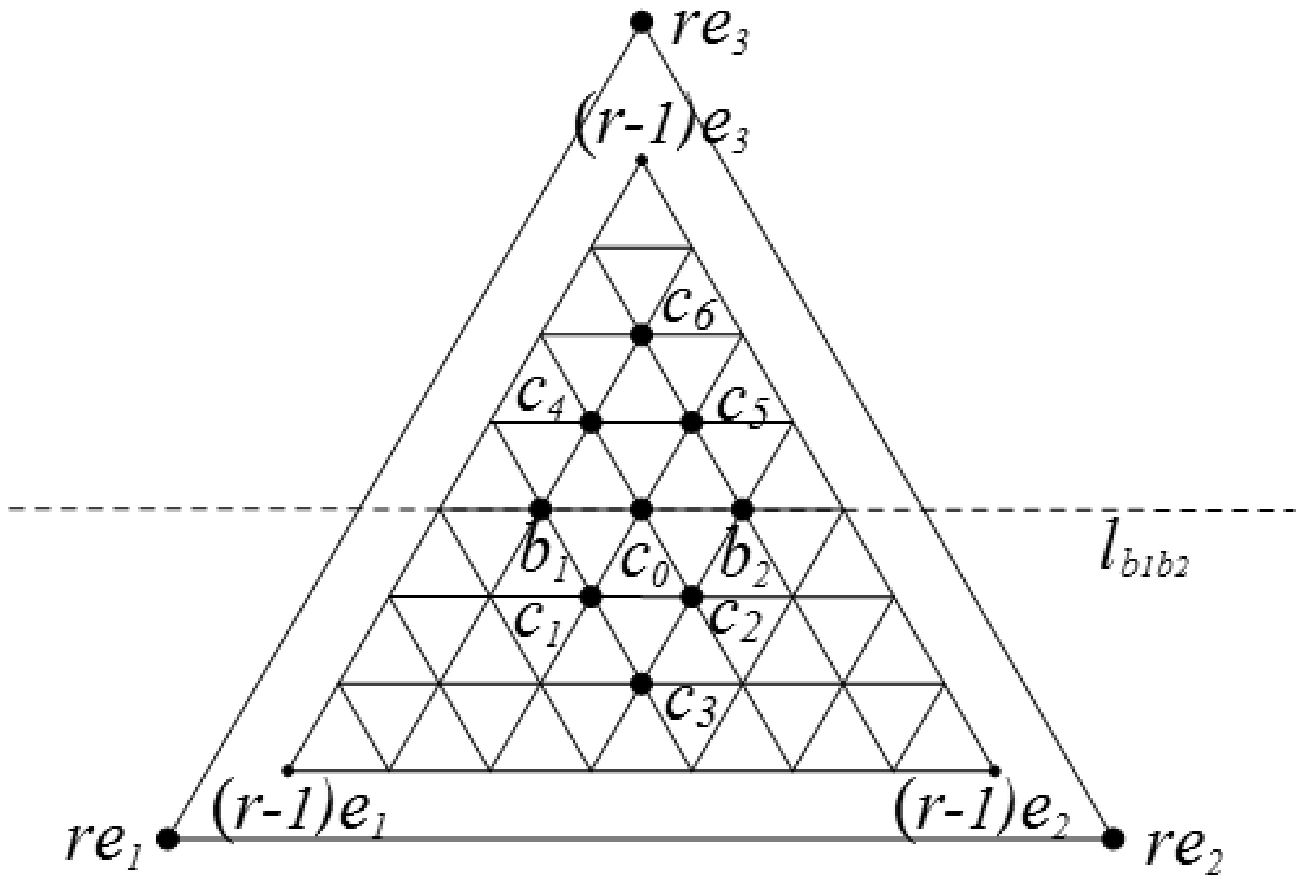}
\includegraphics[height=4.5cm]{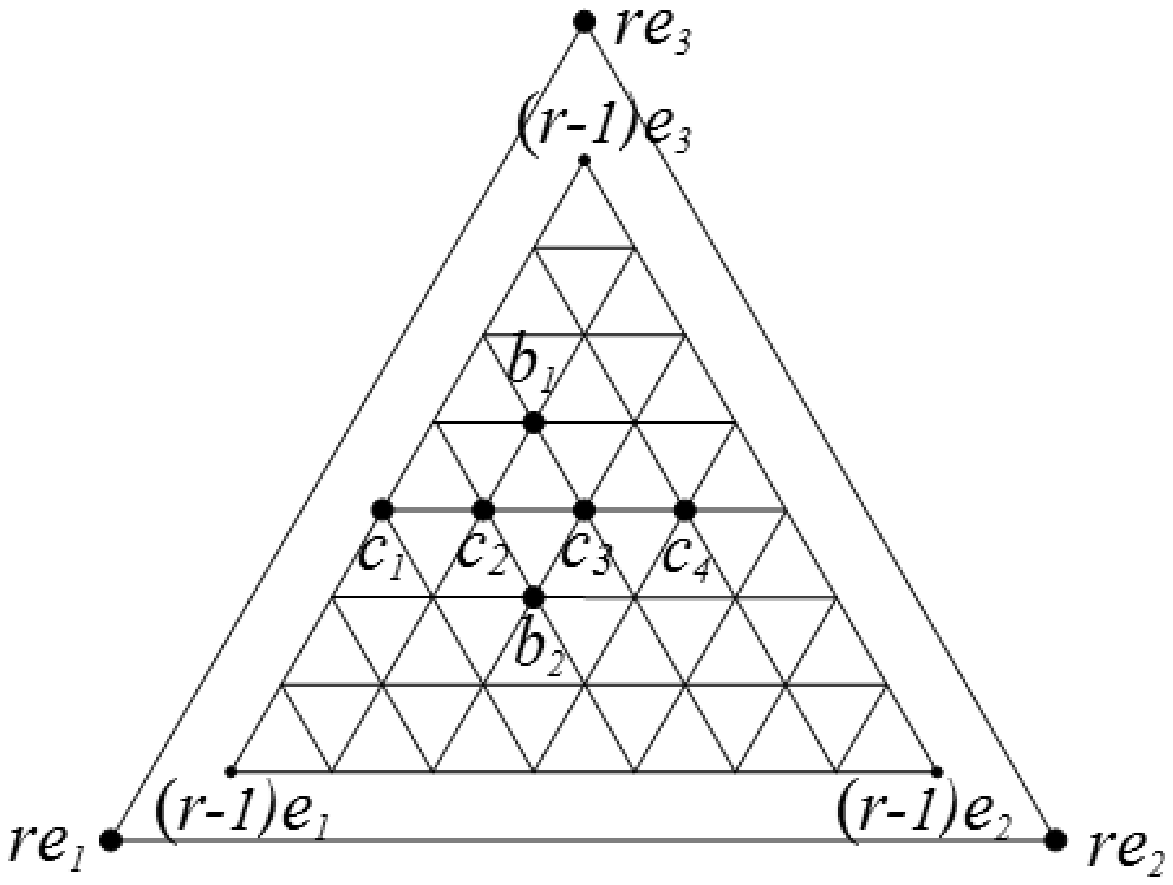}

$(a)$ Case 1.\hspace{4cm} $(b)$ Case 2.
    \caption{$t\geq 3$,
$q=2$.}
    \label{fig:1}           
\end{figure}

Case 1. $\exists \ b_i, b_k \in B$ such that $q_{ik}=2$,
$l_{b_ib_k}$ is parallel to one side of $A(r-1)$.

Assume without loss of generality that  $b_1, b_2 \in B$ such that
$q_{12}=2$, $l_{b_1b_2}\parallel l_{(r-1)e_1,
   (r-1)e_2}\parallel l_{re_1,
   re_2}$ and the vectors $\overrightarrow{b_1b_2}$ and
   $\overrightarrow{re_1,re_2}$ have the same direction. Since
   $q=2$, $B\subset D_{b_1}(2)\cap D_{b_2}(2)$ $=\{b_1, b_2, c_0, c_1, c_2,
   c_3,$ $
   c_4, c_5, c_6\}$, as shown in Figure \ref{fig:1} $(a)$, where
   $c_0\in l_{b_1b_2}$, $\{c_1, c_2, c_3\}\subset H^+_{12}$ and $\{c_4, c_5, c_6\}\subset
   H^-_{12}$. Clearly, $v(\conv\{A(r)\cup \{b_1, b_2, c_1\}\})$ $=v(\conv\{A(r)\cup \{b_1,
   b_2\}\})$ $+v(\conv\{b_1,  c_1, re_1, re_2\})$ $+v(\conv\{b_1,  c_1, b_2,
   re_2\})$ $=(r^2+2r)+r+2=r^2+3r+2.$ Similarly, $v(\conv\{A(r)\cup \{b_1, b_2,
   c_2\}\})$ $=r^2+3r+2.$ As a result, if $B\cap H^+_{12} \neq
   \emptyset$, then $v(P)\geq r^2+3r+2$. If  $B\cap H^+_{12} =
   \emptyset$, then all the possible values of $v(P)$ are:
   $$\begin{array}{l}
 v(\conv\{A(r)\cup \{b_1, b_2,
   c_0\}\})=r^2+2r;\\
v(\conv\{A(r)\cup \{b_1, b_2,
   c_4\}\})=r^2+2r+2;\\
   v(\conv\{A(r)\cup \{b_1, b_2,
   c_5\}\})=r^2+2r+2;\\
   v(\conv\{A(r)\cup \{b_1, b_2,
   c_4, c_5\}\})=r^2+2r+3;\\
    v(\conv\{A(r)\cup \{b_1, b_2,
   c_6\}\})=r^2+2r+4.
    \end{array}$$
Case 2. Otherwise, suppose $b_1, b_2\in B$ such that $l_{b_1b_2}
\perp l_{(r-1)e_1, (r-1)e_2}$. Then we have $B\subset \{b_1, b_2,
c_1, c_2, c_3, c_4\}$ and $B$ contains at most two adjacent points
among $\{c_1, c_2, c_3, c_4\}$, as shown in Figure \ref{fig:1}
$(b)$. Clearly, $ v(\conv\{A(r)\cup \{b_1, b_2,
   c_2\}\})=v(\conv\{A(r)\cup \{b_1, b_2
  \})$ $+v(\conv\{b_1, b_2,$ $ c_2, re_1
  \})$ $=r^2+2r+1$. Similarly, $
v(\conv\{A(r)\cup \{b_1, b_2,
   c_3\}\}) =r^2+2r+1$. Furthermore,  $v(\conv\{A(r)\cup \{b_1, b_2,
   c_1\}\})$ $=
v(\conv\{A(r)\cup \{b_1, b_2,
   c_1, c_2\}\})$ $=v(\conv\{A(r)\cup \{b_2, c_1
  \})$ $+v(\conv\{b_1, b_2,  re_2, re_3
  \})+$ $v(\conv \{b_1, b_2,  c_1, re_3
  \})$ $=(r^2+2r)+r+3=r^2+3r+3$. Similarly, $v(\conv\{A(r)\cup \{b_1, b_2,
   c_4\}\})$ $=
v(\conv\{A(r)\cup \{b_1, b_2,
   c_3, c_4\}\})$ $=r^2+3r+3$. Finally, $v(\conv\{A(r)\cup \{b_1, b_2,
   c_2, c_3\}\})$ $=v(\conv\{A(r)\cup \{b_1, b_2
  \})$ $+v(\conv\{b_1, b_2, c_2, re_1
  \})$ $+v(\conv\{b_1, b_2, c_3, re_2
  \})$ $=r^2+2r+2$.

Combining all the discussions above,  we have the following facts:

if $|B|=1$, then $v(P)=r^2$;

if $|B|=2$, then $v(P)\in \{r^2+kr: k=1, 2, \cdots, r-1\}$;

 if $|B|\geq 3$ and $q=1$, then $v(P)\in
\{r^2+r+1, r^2+2r+1\}$;

 if $|B|\geq 3$ and $q=2$, then either
$v(P)\in [r^2+2r, r^2+2r+4]$ or $v(P)\geq r^2+3r$;

if $|B|\geq 3$ and $q\geq 3$, then $v(P)\geq r^2+3r$.

 If $r\geq 6$,
then $(r^2+3r-1)-(r^2+2r+5)=r-6\geq 0$, which means that  $
[r^2+2r+5, r^2+3r-1]$ is nonempty. Clearly there is no $P \in \p(r)$
satisfying $v(P)\in [r^{2}+2r+5, r^{2}+3r-1]$. The proof is
complete.
\end{proof}

\begin{corollary}
For every $v\in [r^{2}+2r, r^{2}+2r+4]$, there exists a $P \in \p(r)$ such
that $v(P)=v$.
\end{corollary}

\section{Concluding remarks}

We mention further that, as Arnold~\cite{ar} suggests, the paraboloid
\[
D_r=\{x \in \rr_+:x_1^2+\dots+x_{d-1}^2 \le x_d \le r^2\}
\]
can be used to give the lower bound in Corollary~\ref{cor:main} as
the following proof, or rather sketch of a proof, shows. Given $V>0$ with $d!v\in \N$
we are going to construct many non-equivalent convex lattice polytopes of volume $V$.

The integer convex hull $I(D_r)$ of $D_r$ has a vertex
corresponding to each lattice point $z=(z_1,\ldots,z_{d-1},0) \in rB^d\cap \rr_+$,
namely the point $z+\|z\|^2e_d$. Denoting this set of vertices by $X$ we have
$r^{d-1} \ll |X|\ll r^{d-1}$. Also, $\vol I(D_r)$ is of order $r^{d+1}$, and
\[
\vol I(D_r) - \vol I(D_r \setminus X)\ll r^{d-1}
\]
as one can check easily.

Define $r$ by $V=\vol I(D_r \setminus X)$ and note that $V\gg r^{d+1} \gg |X|^{(d+1)/(d-1)}$.  For $Z \subset X$
consider the polytopes $D(Z)=I(D_r \setminus Z)$. This is $2^{|X|}\ge \exp\{c_1 V^{(d-1)/(d+1)}\}$ convex lattice polytopes,
each having volume between $V$ and $V+c_2r^{d-1}$ (where $c_1,c_2>0$ are constants depending only on $d$).

At the vertex $(0,\ldots,0,\lfloor r^2 \rfloor)$ of $D(Z)$ one can place a congruent copy
of $S(\rho)$; here one chooses $\rho \in \N$ to be of order $r^{(d-1)/d}$. We denote this copy by $S(\rho)$ as well.
Given $Z \subset X$ set $m=m(Z)=d! (\vol D(Z)-V)$. Theorem~\ref{th:missed} implies the existence of $P(Z) \in S(\rho)$ with $m(P(Z))=m(Z)$.
We define
\[
D^*(Z)=\left[D(Z)\setminus S(\rho)\right] \cup P(Z),
\]
which is a convex lattice polytope and $\vol D^*(Z)=V$.
One shows again that the number of $D^*(W)$s equivalent to a fixed $D^*(Z)$ is at most $d!$.
We leave the details to the interested reader.

\bigskip
{\bf Acknowledgements.}  Support from ERC Advanced Research Grant no 267165 (DISCONV) and  from
Hungarian National Research Grant K 83767 is acknowledged with thanks.

The second author gratefully acknowledges financial supports by NNSF
of China (11071055); NSF of Hebei Province (A2012205080,
A2013205189); Program for New Century Excellent Talents in
University, Ministry of Education of China (NCET-10-0129); the
project of Outstanding Experts' Overseas Training of Hebei Province.

\bigskip

\noindent
Imre B\'ar\'any \\
R\'enyi Institute of Mathematics,\\
Hungarian Academy of Sciences\\
H-1364 Budapest PoB.~127  Hungary\\
{\tt barany@renyi.hu}\\
and\\
Department of Mathematics\\
University College London\\
Gower Street, London, WC1E 6BT, UK

\medskip
\noindent Liping Yuan\\
College of Mathematics and Information Science,\\
Hebei Normal University,\\
050024 Shijiazhuang, P. R. China.\\
lpyuan88@yahoo.com, lpyuan@mail.hebtu.edu.cn\\
and\\
Hebei Key Laboratory of Computational Mathematics and Applications,
\\
050024 Shijiazhuang,  P. R. China.

\end{document}